\documentclass[11pt,letterpaper]{amsart}

\ifdefined\pdfmapline
  \pdfmapline{=cmex7 TeX-cmex7 <cmex7.pfb}
  \pdfmapline{=cmex8 TeX-cmex8 <cmex8.pfb}
  \pdfmapline{=cmex9 TeX-cmex9 <cmex9.pfb}
\fi
\usepackage{amsmath,amssymb,amsfonts,mathtools}
\usepackage{xcolor}
\usepackage{enumitem}
\usepackage{hyperref}
\usepackage{cleveref}
\usepackage{mathscinet}

\calclayout
\DeclareFontFamily{U}{matha}{\hyphenchar\font45}
\DeclareFontShape{U}{matha}{m}{n}{
      <5> <6> <7> <8> <9> <10> gen * matha
      <10.95> matha10 <12> <14.4> <17.28> <20.74> <24.88> matha12
      }{}
\DeclareSymbolFont{matha}{U}{matha}{m}{n}
\DeclareFontSubstitution{U}{matha}{m}{n}
\DeclareFontFamily{U}{mathx}{\hyphenchar\font45}
\DeclareFontShape{U}{mathx}{m}{n}{
      <5> <6> <7> <8> <9> <10>
      <10.95> <12> <14.4> <17.28> <20.74> <24.88>
      mathx10
      }{}
\DeclareSymbolFont{mathx}{U}{mathx}{m}{n}
\DeclareFontSubstitution{U}{mathx}{m}{n}
\DeclareMathDelimiter{\vvvert} {0}{matha}{"7E}{mathx}{"17}%

\makeatletter
\def\letterdef#1#2#3{\def\letterdef@##1{\expandafter\def\csname #1\endcsname{#2}}%
  \letterdef@@#3{?\@car{}}\@nil}
\def\letterdef@@#1{\@gobble#1\letterdef@{#1}\letterdef@@}
\makeatother
\DeclarePairedDelimiterX{\klx}[2]{(}{)}{%
  #1\;\delimsize\|\;#2%
}
\DeclarePairedDelimiterX{\quantklx}[3]{(}{)}{%
  #1\;\delimsize\|\;#2\;\delimsize\vert\;#3%
}
\DeclarePairedDelimiterX{\inner}[2]{\langle}{\rangle}{%
  #1,#2%
}
\newcommand{\R}{\mathbf R} 
\letterdef{c#1}{\mathcal{#1}}{ABCDEFGHIJKLMNOPQRSTUVWXYZ}
\letterdef{b#1}{\mathbb{#1}}{ABCDEFGHIJKLMNOPQRSTUVWXYZ}
\letterdef{bf#1}{\mathbf{#1}}{ABCDEFGHIJKLMNOPQRSTUVWXYZ}

\newcommand{\ud}[0]{\,\mathrm{d}}  
\newcommand{\1}{\mathbf 1} 
\let\ones\1
\newcommand{\half}{\frac12} 
\let\epsilon\varepsilon
\newcommand{\eps}{\varepsilon}

\let\preceq\preccurlyeq

\renewcommand{\leq}{\leqslant}
\renewcommand{\geq}{\geqslant}
\newcommand{\argmin}{\mathop{\rm arg\,min}}

\DeclareMathOperator{\diam}{diam}

\newcommand{\kl}{\mathrm{KL}\klx}

\makeatletter
\newcommand{\E}{\operatorname*{\mathbf{E}}\ilimits@}
\makeatother
\makeatletter
\renewcommand{\P}{\operatorname*{\mathbf{P}}\ilimits@}
\makeatother

\newcommand{\ie}{\textit{i}.\textit{e}., }
\newcommand{\eg}{\textit{e}.\textit{g}., }

\newcounter{algorithmctr}
\renewcommand{\thealgorithmctr}{\arabic{algorithmctr}}
   {\refstepcounter{algorithmctr}\begin{list}{}{%
       \setlength{\rightmargin}{0\linewidth}%
       \setlength{\leftmargin}{0\linewidth}}%
       \rmfamily\small
       \item[]{\setlength{\parskip}{0ex}\hrulefill\par%
        \nopagebreak{\bfseries\textsf{Algorithm \thealgorithmctr~}}}}%
   {{\setlength{\parskip}{-1ex}\nopagebreak\par\hrulefill} \end{list}}

\makeatletter
\long\def\@makecaption#1#2{
        \vskip 0.8ex
        \setbox\@tempboxa\hbox{\small {\bf #1.} #2}
        \parindent 1.5em 
        \dimen0=\hsize
        \advance\dimen0 by -3em
        \ifdim \wd\@tempboxa >\dimen0
                \hbox to \hsize{
                        \parindent 0em
                        \hfil 
                        \parbox{\dimen0}{\def\baselinestretch{0.96}\small
                                {\bf #1.} #2
                                } 
                        \hfil}
        \else \hbox to \hsize{\hfil \box\@tempboxa \hfil}
        \fi
        }
\makeatother

\theoremstyle{plain}
\newtheorem{theorem}{Theorem}
\newtheorem{proposition}[theorem]{Proposition}
\newtheorem{lemma}[theorem]{Lemma}

\theoremstyle{definition}

\newtheoremstyle{boldromanremark}
  {\dimexpr\topsep+0.4\baselineskip\relax}
  {\dimexpr\topsep+0.4\baselineskip\relax}
  {\normalfont}{}
  {\bfseries}
  {.}{.5em}
  {}
\theoremstyle{boldromanremark}
\newtheorem{remark}{Remark}

\makeatletter
\renewenvironment{proof}[1][\proofname]{\par
  \pushQED{\qed}%
  \normalfont \topsep6\p@\@plus6\p@\relax
  \trivlist
  \item[\hskip\labelsep
        \bfseries\itshape
    #1\@addpunct{.}]\ignorespaces
}{%
  \popQED\endtrivlist\@endpefalse
}
\makeatother

\newcommand{\klconst}{C_{\rm KL}}
\newcommand{\BayesOpt}[2]{\mathfrak{B}_{#1}{(#2)}}
\newcommand{\AllMeasures}[1]{\cP(#1)}
\newcommand{\LiuFunctional}[1]{\mathfrak{L}(#1)}
\newcommand{\Normal}[2]{N(#1,#2)}
\title[Remark on Majorizing Measures]{A remark on the \\ majorizing measures theorem for general processes}

\author{Reese Pathak}
\thanks{RP gratefully acknowledges support from the NSF under grant DMS-2503579.}
\address{Department of Statistics, UC Berkeley,\and\newline\hspace*{\parindent}%
School of Oper. Res. \& Inf. Engineering (ORIE), Cornell University}
\email{pathakr@berkeley.edu}

\author{Nikita Zhivotovskiy}
\address{Department of Statistics, 
UC Berkeley}
\email{zhivotovskiy@berkeley.edu}

\date{\today}

\begin{document}

\begin{abstract}
We show that the lower bound in the majorizing measures theorem holds for a large class of random vectors. Specifically, suppose $X \sim \mu$ is a centered random vector in $\R^n$ with
\[
\klconst(\mu) = \sup_{\substack{\theta, \eta \in \R^n \\ \theta \neq \eta}} 
\frac{\kl{\mu_\theta}{\mu_\eta}}{\|\theta - \eta\|_2^2} < \infty,
\]
where $\mu_\theta$ denotes the law of the translate $\theta + X$.
Then, for every nonempty, bounded $T \subset \R^n$, 
\[
\sqrt{\klconst(\mu)}\, \E_\mu \Big[\sup_{t \in T} \, \langle X, t \rangle \Big] 
\gtrsim \gamma_2(T), 
\]
where the righthand side denotes Talagrand's generic chaining functional. 
This result recovers, as a special case, the lower bound in the majorizing measures theorem for centered Gaussian 
processes.  Our argument  critically relies on  the rate-distortion integral, recently  introduced 
by J.\ Liu~\cite{Liu26}.
\end{abstract}

\maketitle

\section{Introduction}

A central result in the theory of Gaussian processes is 
the Fernique-Talagrand majorizing measures theorem~\cite{Fer75,Tal87}. 
For a metric space $(T, d)$, 
Talagrand's generic chaining functional is,
\begin{equation}
\label{eqn:generic-chaining-functional}
\gamma_2(T, d) = \inf \sup_{t \in T} \sum_{n=0}^\infty 2^{n/2} \diam(A_n(t)).
\end{equation}
The infimum ranges over \emph{admissible sequences},
\ie increasing partitions $\{\cA_n\}_{n \geq 0}$ of $T$ 
such that $|\cA_0| = 1$ and $|\cA_n| \leq 2^{2^n}$ for $n \geq 1$.
Above, $A_n(t)$ denotes the element of the 
partition $\cA_n$ which contains $t$.
For a set $T \subset \R^n$ equipped with the Euclidean metric, 
\ie $d(t,s) = \|t - s\|_2$, 
we denote the generic chaining functional by $\gamma_2(T)$.

A modern formulation of the majorizing measures theorem (\eg~\cite[Theorem~2.10.1]{Tal21}) states that 
for a centered Gaussian process $\{G_t\}_{t \in T}$ with 
$d(t, s) = (\E (G_t - G_s)^2)^{1/2}$, it holds that
\begin{equation}
\label{eqn:majorizing-measures-theorem}
\frac{1}{C} \, \gamma_2(T, d) \leq \E \sup_{t \in T} G_t \leq C \, \gamma_2(T, d),
\end{equation}
where $C > 0$ is a universal constant. 
We refer readers to~\cite{Tal21} for context and history of this result.

The main result of this note is that the lower bound in the majorizing 
measures theorem holds for random vectors $X \in \R^n$, which 
satisfy a quadratic growth condition on the Kullback-Leibler (KL) divergence\footnote{For two probability measures $P, Q$, is defined by $\kl{P}{Q} = \int \log \tfrac{\ud P}{\ud Q}(x) \, \ud P(x)$ if $P \ll Q$ and $+\infty$ otherwise.}
between the law of translates of $X$. We defer discussion of the result to~\Cref{sec:remarks-on-main-result}.
\begin{theorem}
\label{thm:quadratic-kl-to-gamma-2}
  Suppose that $X \sim \mu$ is a centered random vector in $\R^n$ such that
  \begin{equation}
  \label{ineq:quadratic-KL-assumption}
  \klconst(\mu) = \sup_{\substack{\theta, \eta \in \R^n \\ \theta \neq \eta}} 
\frac{\kl{\mu_\theta}{\mu_\eta}}{\|\theta - \eta\|_2^2} < \infty,
  \end{equation}
  where $\mu_\theta$ denotes the law of the translate $\theta + X$.
  Then, for any nonempty, bounded $T \subset \R^n$,
  \[
  \E_\mu \Big[\sup_{t \in T} \, \langle X, t \rangle \Big] \gtrsim \frac{1}{\sqrt{\klconst(\mu)}} \, \gamma_2(T).
  \]
\end{theorem}

\section{\texorpdfstring{Proof of~\Cref{thm:quadratic-kl-to-gamma-2}}{Proof of the main result}}

We assume, without loss of generality, that $T \subset \R^n$ is closed, convex and $0 \in T$.
We define
\[
w_\mu(T) = \E_\mu \Big[\sup_{t \in T} \, \langle X, t \rangle \Big].
\]

\subsection{Step 1: Reduction to priors}

The first step of our argument shows that the $\mu$-width $w_\mu(T)$ can be bounded from below in terms of a variational problem involving measures.

Let $\AllMeasures{T}$ denote all probability measures supported on $T$.
For $\pi \in \AllMeasures{T}, \sigma > 0$, define 
\begin{equation}
\label{eqn:bayes-risk}
\varepsilon_\mu(\sigma, \pi) = 
\inf_{\hat \theta} \Big(\E_{\theta \sim \pi} 
\E_{X \sim \mu} \Big[\|\hat \theta(\theta + \sigma X) - \theta\|_2^2\Big]\Big)^{1/2}.
\end{equation}
The infimum ranges over measurable mappings 
$\hat \theta \colon \R^n \to \R^n$ and is attained with the posterior mean, $\hat \theta(Y_\sigma) = \E[\theta \mid Y_\sigma]$ where 
$Y_\sigma = \theta + \sigma X$.
We then define the quantity 
\[
\BayesOpt{\mu}{T} =  
\sup_{\pi \in \AllMeasures{T}} 
\int_0^\infty \frac{\varepsilon_\mu(\sigma, \pi)^2}{\sigma^2} \, \ud \sigma.
\]
\begin{lemma}
\label{lem:Bayes-lower}
    It holds that $w_\mu(T) \geq \half \, \BayesOpt{\mu}{T}$.
\end{lemma}
\begin{proof}
    Fix $\theta \in T$ and $\pi \in \AllMeasures{T}$. Since $X \sim \mu$ is centered, it holds that 
    \begin{equation}
    \label{eqn:translation-invariance}
    w_\mu(T) = w_\mu(T - \theta), 
    \quad \mbox{for any}~\theta \in T.
    \end{equation}
    Denote the projection onto a nonempty, closed convex set $K \subset \R^n$ by
    \[
    \Pi_K(x) = \argmin_{y \in K} \, \|x - y\|_2, \quad 
    \mbox{for}~x\in \R^n.
    \]
    It is immediate from the definition that for any $r > 0$, and any $x, h \in \R^n$ it holds that
    \[
    \Pi_{K}(x + h) = h + \Pi_{K - h}(x) \quad \mbox{and} \quad 
    \Pi_K(rx) = r \, \Pi_{r^{-1}K}(x).
    \]
    Thus, applying~\cite[Theorem~2.2]{PatZhi26} to  $T - \theta$ 
    yields,
    \begin{equation}
    \label{eqn:projection-identity}
    2\,  w_\mu(T-\theta) =  \int_0^\infty 
    \E_\mu \|\Pi_{\tfrac{T-\theta}{\sigma}}(X)\|_2^2 \, \ud \sigma 
    =  
    \int_0^\infty 
    \frac{\E_\mu \|\Pi_{T}(\theta + \sigma X) - \theta\|_2^2}{\sigma^2} \, \ud \sigma.
    \end{equation}
    Integrating with respect to $\pi$ and applying Fubini's theorem, we obtain from~\cref{eqn:translation-invariance,eqn:projection-identity} that 
    \[
    w_\mu(T) = 
    \E_{\theta \sim \pi}
    w_{\mu}(T - \theta) 
    = 
    \half \int_0^\infty 
    \frac{\E_{\theta \sim \pi} \E_{X \sim \mu} \|\Pi_{T}(\theta + \sigma X) - \theta\|_2^2}{\sigma^2} \, \ud \sigma
    \geq 
    \half 
    \int_0^\infty 
    \frac{\eps_\mu(\sigma, \pi)^2}{\sigma^2} \, \ud \sigma.
    \]
    Passing to the supremum over $\pi \in \AllMeasures{T}$ yields the claim.
\end{proof}

\subsection{Step 2: Lower bound via Liu's distortion integral}

We now further lower bound $\BayesOpt{\mu}{T}$, 
using the rate-distortion integral introduced in~\cite{Liu26}. 
We make use of standard notions from information theory (\eg mutual information, conditional mutual information, data processing inequality); we refer the reader to~\cite{PolWu25} for details and definitions.

\subsubsection{Liu's construction} 
We now describe the construction in~\cite[Definition~3]{Liu26}.
For $r \geq 0$ and $\pi \in \AllMeasures{\R^n}$ define the
\emph{self-coupling rate distortion},
\[
i_\pi(r) =
\inf\Big\{\, I(Z;\widehat Z) :
Z,\widehat Z \sim \pi,\ 
\E \|Z - \widehat Z\|_2^2 \leq r^2
\,\Big\}.
\]
To be clear, the infimum ranges over all couplings of the pair $(Z, \widehat Z)$, 
subject to the constraints.
We then consider \emph{Liu's distortion integral},
\[
\LiuFunctional{T} =
\sup_{\pi \in \AllMeasures{T}}
\int_0^{\diam(T)} \sqrt{i_\pi(r)} \, \ud r.
\]

\subsubsection{From measures to Liu's distortion integral}

We now lower bound $\BayesOpt{\mu}{T}$ in terms of $\LiuFunctional{T}$.

\begin{lemma}
\label{lem:bayes-to-liu}
    For any $\pi \in \AllMeasures{T}$, it holds that 
    \begin{equation}
    \label{ineq:final-lower-bound-on-integrated-bayes-risk}
    \sqrt{\klconst(\mu)}
    \int_0^\infty
    \frac{\varepsilon_\mu(\sigma,\pi)^2}{\sigma^2}
    \, \ud \sigma \geq
    \frac{1}{4} \, 
    \int_0^\infty \sqrt{i_\pi(r)} \, \ud r.
    \end{equation}
    Consequently,
     $\sqrt{\klconst(\mu)} \, \BayesOpt{\mu}{T} \gtrsim \LiuFunctional{T}$.
\end{lemma}
\begin{proof}
Let $\theta \sim \pi$ and $X \sim \mu$ be independent. Set
$Y_\sigma = \theta + \sigma X$.
For any $\eta, \eta' \in \R^n$ and any $\sigma > 0$, 
\begin{equation}
\label{ineq:scaled-KL-assumption}
\kl{\cL(\eta+\sigma X)}{\cL(\eta'+\sigma X)}
= 
\kl{\mu_{\eta/\sigma}}{\mu_{\eta'/\sigma}}
\leq
\klconst(\mu)\frac{\|\eta-\eta'\|_2^2}{\sigma^2}.
\end{equation}
We claim that for every $r>0$ and $\sigma>0$,
\begin{equation}
\label{ineq:rate-distortion-bounds}
i_\pi(\sqrt{2} \varepsilon_\mu(\sigma, \pi)) 
\stackrel{{\rm (i)}}\leq 
I(\theta;Y_\sigma), 
\quad \mbox{and} \quad 
I(\theta;Y_\sigma)
\stackrel{{\rm (ii)}}\leq
i_\pi(2r)
+4\klconst(\mu)\frac{r^2}{\sigma^2}.
\end{equation}
Assuming~\eqref{ineq:rate-distortion-bounds}, the result follows 
quickly. Define the threshold 
\[
\sigma_\mu(r, \pi) = \frac{2\sqrt{\klconst(\mu)}\,r}
{\sqrt{i_\pi(r)-i_\pi(2r)}}.
\]
By~\cref{ineq:rate-distortion-bounds}, whenever
$i_\pi(r)>i_\pi(2r)$ and $\sigma>\sigma_\mu(r, \pi)$,
we have
$\sqrt{2} \varepsilon_\mu(\sigma,\pi)\geq r$. 
Hence, 
\begin{multline}
\label{ineq:lower-bound-on-integrated-bayes-risk}
\int_0^\infty
\frac{\varepsilon_\mu(\sigma,\pi)^2}{\sigma^2}
\ud\sigma
= 
\int_0^\infty\!\int_0^{\infty}
\frac{r}{\sigma^2}\1_{\{\sqrt{2}\varepsilon_\mu(\sigma,\pi)\geq r\}}
\,\ud r\,\ud \sigma
 \geq
\int_0^\infty\!\int_{\sigma_\mu(r, \pi)}^{\infty}
\frac{r}{\sigma^2}\1_{\{i_\pi(r)>i_\pi(2r)\}}
\,\ud \sigma\,\ud r
\\\hspace{11em}=
\int_0^\infty \frac{r}{\sigma_\mu(r, \pi)} \1_{\{i_\pi(r) > i_\pi(2r)\}} \, \ud r
=
\frac{1}{2\sqrt{\klconst(\mu)}}
\int_0^\infty
\sqrt{i_\pi(r)-i_\pi(2r)}
\,\ud r.
\end{multline}
Since $i_\pi$ is nonincreasing and $\lim_{r\to\infty} i_\pi(r)=0$, 
it holds that 
\[
i_\pi(r)
=
\sum_{k = 0}^\infty
\big(i_\pi(2^k r)-i_\pi(2^{k+1}r)\big), 
\quad \mbox{for all}~r>0.
\]
The subadditivity of $\sqrt{\cdot}$ and the change of variables $u = 2^k r$ yields: 
\begin{equation}
\label{ineq:distortion-integral-upper-bound}
\int_0^\infty \sqrt{i_\pi(r)}\,\ud r
\leq
\sum_{k=0}^\infty  \frac{1}{2^k}
\int_0^\infty
\sqrt{i_\pi(u)-i_\pi(2u)}
\,\ud u 
=
2
\int_0^\infty
\sqrt{i_\pi(u)-i_\pi(2u)}
\,\ud u.
\end{equation}
Combining inequalities~\eqref{ineq:lower-bound-on-integrated-bayes-risk} and~\eqref{ineq:distortion-integral-upper-bound} 
establishes~\eqref{ineq:final-lower-bound-on-integrated-bayes-risk}. 

Hence, we turn to the bounds~\eqref{ineq:rate-distortion-bounds}. 
Indeed, let
$\hat \theta=\hat \theta(Y_\sigma) = \E [\theta \mid Y_\sigma]$. 
Let $\widetilde \theta$ be conditionally independent of $\theta$ given
$\hat \theta$, with $\cL(\widetilde \theta \mid \hat \theta)
=\cL(\theta \mid \hat \theta)$. 
Then $(\theta,\widetilde \theta)$ is a self-coupling of $\pi$, and
\[
\E\|\theta-\widetilde \theta\|_2^2
=
2\E\|\theta-\E[\theta\mid \hat \theta]\|_2^2
= 
2 \E_{\hat \theta} \inf_{a \in \R^n} \E[\|\theta - a\|_2^2\mid \hat \theta]
\leq  
2\E\|\theta-\hat \theta\|_2^2
=2 \varepsilon_\mu(\sigma, \pi)^2.
\]
Hence, the data processing inequality yields~\cref{ineq:rate-distortion-bounds}(i):
\[
i_\pi(\sqrt{2} \varepsilon_\mu(\sigma, \pi)) \leq I(\theta;\widetilde \theta)
\leq
I(\theta;\hat \theta)
\leq
I(\theta;Y_\sigma).
\]
For the second inequality, fix $\eta>0$. 
There exists a self-coupling $(\theta,\theta')$ of $\pi$
such that
\begin{equation}
\label{ineq:self-coupling-bound}
\E\|\theta-\theta'\|_2^2 \leq 4r^2,
\qquad
I(\theta;\theta')\leq i_\pi(2r)+\eta.
\end{equation}
Let $X \sim \mu$ be independent of $(\theta,\theta')$. By the chain rule,
\begin{equation}
\label{eqn:mutual-information-chain-rule}
I(\theta;Y_\sigma)
\leq
I(\theta;\theta',Y_\sigma)
=
I(\theta;\theta')+I(\theta;Y_\sigma\mid \theta').
\end{equation}
Now note that by~\cref{ineq:scaled-KL-assumption,,ineq:self-coupling-bound}, 
\begin{multline}
\label{ineq:upper-bound-on-mutual-information-conditional-on-theta'}
I(\theta; Y_\sigma \mid \theta')
= \E_{\theta'} \Big[ \inf_{Q \in \AllMeasures{\R^n}} 
\E_{\theta \mid \theta'}
\kl{\cL(\theta + \sigma X)}{Q} \Big] 
\\ \leq
\E_{\theta, \theta'} \Big[ 
\kl{\cL(\theta + \sigma X)}{\cL(\theta' + \sigma X)} \Big]
\leq 
\klconst(\mu)\frac{\E\|\theta-\theta'\|_2^2}{\sigma^2}
\leq 4\klconst(\mu)\frac{r^2}{\sigma^2}.
\end{multline}
Combining~\cref{ineq:self-coupling-bound,,eqn:mutual-information-chain-rule,ineq:upper-bound-on-mutual-information-conditional-on-theta'},
and letting $\eta \downarrow 0$ proves~\eqref{ineq:rate-distortion-bounds}(ii).
\end{proof}

\subsection{Step 3: Comparison to the generic chaining functional and conclusion}

The final step of the argument is to relate the distortion integral 
to the generic chaining functional; it is implicit in~\cite{Liu26}. 
The details are given below, borrowing essentially all of the 
main ideas from~\cite{Liu26}.

\begin{theorem}[{\cite{Liu26}}]
\label{thm:Liu-lower}
    It holds that $\LiuFunctional{T} \gtrsim \gamma_2(T)$.
\end{theorem}
\begin{proof}
Define for a set $S \subset \R^n$, the quantity 
\[
\Gamma(S) = \inf_{m \in \AllMeasures{S}}
\sup_{t \in S} I_m(t),
\quad
I_m(t) =
\int_0^{\diam(S)} \sqrt{\log\frac{1}{m(B(t,r))}} \, \ud r.
\]
Here $B(t,r)=\{u\in S:\|u-t\|_2\leq r\}$.
It suffices to show that $\LiuFunctional{F} \gtrsim \Gamma(F)$ for
all finite subsets $F \subset T$; 
this follows from~\cite[\S 6]{Liu26}; 
the details are given below as~\Cref{prop:liu-functional-controls-gamma}.
Then,
\[
\LiuFunctional{T} \geq
\sup_{\substack{F \subset T \\ |F|<\infty}} 
\LiuFunctional{F} 
\gtrsim \sup_{\substack{F \subset T \\ |F|<\infty}} \Gamma(F)
\stackrel{\rm(i)}\gtrsim \sup_{\substack{F \subset T \\ |F|<\infty}} \gamma_2(F) 
\stackrel{\rm(ii)}\gtrsim \gamma_2(T). 
\]
Note that inequalities (i) and (ii) 
are standard and do not rely on the use 
of the majorizing measures theorem; for proofs 
see~\cite[\S3.3.3]{Tal21} and~\cite[Theorem~1.3.6(c)]{Tal05}, respectively.
\end{proof}

To conclude, we simply apply~\Cref{lem:Bayes-lower,lem:bayes-to-liu} and~\Cref{thm:Liu-lower} in succession:
\[
\sqrt{\klconst(\mu)} \, w_\mu(T) \gtrsim \sqrt{\klconst(\mu)} \, \BayesOpt{\mu}{T}
\gtrsim 
\LiuFunctional{T}
\gtrsim 
\gamma_2(T).
\]

\section{Some remarks on the main result}
\label{sec:remarks-on-main-result}

We briefly discuss some important aspects of~\Cref{thm:quadratic-kl-to-gamma-2}.

\begin{remark}[The Gaussian case]  
Let $G \sim \Normal{0}{I_n}$; 
it is easy to check that 
$\klconst(\Normal{0}{I_n}) = \tfrac{1}{2}$. In this case,~\Cref{thm:quadratic-kl-to-gamma-2} implies that there exists a constant $C > 0$ such that
\begin{equation}
\label{eqn:canonical-lower-bound}
\E_{G \sim \Normal{0}{I_n}}\Big[ \sup_{x \in K} \, \langle G, x \rangle\Big] \geq C \, \gamma_2(K)
\quad \mbox{for all}~n \geq 1,~\mbox{all nonempty, bounded}~K \subset \R^n.
\end{equation}
It is well known that inequality~\eqref{eqn:canonical-lower-bound} implies (\eg by finite approximation) the inequality
\begin{equation}
\label{ineq:gaussian-mmt-lower}
\E \sup_{t \in T} G_t \gtrsim \gamma_2(T, d),
\end{equation}
for any centered Gaussian process $\{G_t\}_{t \in T}$ with canonical metric  
$d(t, s) = (\E (G_t - G_s)^2)^{1/2}$. This is precisely the lower bound in the majorizing measures theorem; compare with~\eqref{eqn:majorizing-measures-theorem}.
\end{remark}

\begin{remark}[Measurability concerns]
We do not dwell on measurability issues here; however, note that when $T \subset \R^n$ is bounded the map $x \mapsto \sup_{t \in T} \, \langle x, t\rangle$ is finite everywhere and Lipschitz continuous. Hence $\sup_{t \in T} \, \langle X, t \rangle$
is a measurable function of a random vector $X$.
\end{remark}

\begin{remark}[The role of priors]
As discussed in~\cite{PatZhi26},
the statistical properties of the Gaussian sequence model (\ie the family of laws, $\theta + \sigma G$ for $G \sim \Normal{0}{I_n}$ and $\theta \in T$) 
can be related to the Gaussian width (and, more generally, the $\mu$-width by replacing $G$ with a general random vector $X \sim \mu$).
This perspective leads to an information-theoretic proof of the Sudakov minoration, \eg \cite[\S 5.2]{PatZhi26}. There, the authors employed the statistical minimax rate, which essentially corresponds to picking a family of worst-case priors $\pi_\sigma \in \AllMeasures{T}$, which may depend on $\sigma > 0$. 
This approach produces chaining bounds analogous to Dudley's entropy integral~\cite[pp.~21-22]{PatZhi26}. In contrast, here we apply the integral identity~\cite[Theorem 2]{PatZhi26} for a prior $\pi$, simultaneously for \emph{all} $\sigma > 0$. This method is suggested by Fernique's formulation of
majorizing measures~\cite{Fer75}, 
and was employed in~\cite{Liu26,Zad26}.\footnote{We pursued~\Cref{thm:quadratic-kl-to-gamma-2} after the paper~\cite{Zad26} was posted.}
\end{remark}

\begin{remark}[Comparison with~\cite{Zad26}]
In the Gaussian setting (\ie $\mu = \Normal{0}{I_n}$), the recent work by I.\ Zadik~\cite[\S 2.3]{Zad26} presented a proof of the lower bound in the majorizing measures theorem, based on Bayesian arguments (essentially invoking~\cite[Theorem~2.2]{PatZhi26} as in Step 1 above, but for finite $T$). In~\cite{Zad26}, Step 2 in our argument is replaced by using the following identity which makes heavy use of Gaussianity: $\sigma^3 \partial_\sigma I(\theta; Y_{\sigma}) = - \varepsilon_\mu(\sigma, \pi)^2$, translated to our notation. An analogue of~\Cref{lem:bayes-to-liu} is obtained by leveraging information-theoretic identities~\cite[Lemmas 2.4, 2.5, 2.6]{Zad26}. For general measures, this argument has to be modified: $\sigma \mapsto I(\theta; Y_\sigma)$ need not be differentiable in general. Nonetheless,
as we showed,~\Cref{lem:bayes-to-liu} follows from the substantially weaker KL condition,~\cref{ineq:quadratic-KL-assumption}.
\end{remark}

\begin{remark}[Comparison with~\cite{Liu26}]
The paper by J.\ Liu~\cite{Liu26} also contains a proof of the majorizing measures theorem in the Gaussian setting. The proof 
is quite different: it uses a ``lifting'' procedure~\cite[\S 5]{Liu26}, which compares a non-stationary to a (higher-dimensional) stationary process, for which Dudley's entropy integral is known to be sharp~\cite{Fer75}. Due to this key Gaussian input, we do not expect this proof strategy to generalize easily under the weaker KL condition~\eqref{ineq:quadratic-KL-assumption}.
\end{remark}

\begin{remark}[The $\log$-smooth setting]
A sufficient condition for~\eqref{ineq:quadratic-KL-assumption} to hold is that $\mu$ is $\log$-smooth. Specifically, suppose that $\mu$ is a centered probability measure on $\R^n$ with density $\ud \mu = \exp(-V(x)) \, \ud x$ for some $V \in C^2(\R^n)$. If for any $x \in \R^n$, it holds that $\nabla^2 V(x) \preceq \alpha I_n$ for some $\alpha > 0$, then $\klconst(\mu) \leq \tfrac{\alpha}{2}$.
\Cref{thm:quadratic-kl-to-gamma-2} implies that 
$w_\mu(T) \gtrsim \tfrac{1}{\sqrt{\alpha}} \gamma_2(T)$ for all nonempty, bounded $T \subset \R^n$.
This inequality can also be derived from~\cite[Theorem~1.2]{Har04}; we omit the details. The condition $\klconst(\mu) <\infty$ is 
weaker: if $\tfrac{\ud \mu}{\ud x} = \tfrac{\exp(-\sum_{i=1}^n |x_i| )}{2^n}$, then 
$\klconst(\mu) \lesssim 1$, but $\mu$ is not $\log$-smooth.
\end{remark}

\appendix

\section{Liu's lower bound argument for the distortion integral}

Here we give the proof of~\Cref{prop:liu-functional-controls-gamma}, based on the ideas 
in~\cite{Liu26}.  

\begin{proposition}[{\cite{Liu26}}]
\label{prop:liu-functional-controls-gamma}
For a nonempty, finite set $F \subset \R^n$, it holds that 
$\LiuFunctional{F} \gtrsim \Gamma(F)$.
\end{proposition}
\begin{proof}
  If $\diam(F) = 0$ there is nothing to prove; suppose $\diam(F) > 0$. Since $i_\pi(r)=0$ for $r\geq \diam(F)$, we may integrate to $\infty$ in the definition 
  of $\LiuFunctional{F}$.
  Liu introduces, for $m \in \AllMeasures{F}$ and $t \in F$, the quantity (see~\cite[Eq.~(81)]{Liu26}),
  \[
  \Phi_m(t,\alpha)
  =
  \inf_{\nu \in \AllMeasures{F}}
  \Big\{\, 
  \frac{1}{\alpha^2}\int \|t-u\|_2^2\,\ud \nu(u)
    + \kl{\nu}{m}
  \, \Big\}.
  \]
  Liu's penalization lemma
  \cite[Lemma~7, eqn.~(80)]{Liu26}, followed 
by the application of the Sion minimax theorem~
\cite[Lemma~7, eqn.~(81-82)]{Liu26} yields 
  \begin{equation}\label{eqn:liu-functional-controls-gamma}
  \LiuFunctional{F}
  \geq \frac{1}{4} \, 
  \inf_{m\in \AllMeasures{F}}
  \max_{t\in F}
  \int_0^\infty \Phi_m(t,\alpha)\,\ud\alpha.
  \end{equation}
  The conclusion follows from Liu's data processing estimate.  
  Specifically, for $a > 1, p \in (0, 1)$ define: 
  \[
  c_1(a) = \frac{\sqrt{a^2 - 1}}{a^2}, \quad 
  c_2(a) = \sqrt{\frac{a^2}{a^2 - 1} \, h\Big(\frac{1}{a^2}\Big)}, 
  \quad \mbox{and} \quad 
  h(p) = -p \log p - (1-p) \log (1-p).
  \]
  Then from \cite[Corollary~1, eqn. (83---87)]{Liu26}, it holds for any 
  $m \in \AllMeasures{F}$ and $t \in F$ that 
  \[
  \int_0^\infty \Phi_m(t,\alpha)\,\ud \alpha
  \geq
  2 \, \sup_{a > 1} c_1(a) \Big\{ I_m(t) - c_2(a)\diam(F)\Big\}, 
  \]
  Moreover, it holds that $\Gamma(F)\gtrsim \diam(F)$. 
  Indeed, pick $s, t \in F$ such that
  $\|s-t\|_2=\diam(F)$. For $r \in (0,\tfrac{\diam(F)}{2})$, 
  we have $B(s, r) \cap B(t, r) = \emptyset$. Consequently,
  let $p(r) = m(B(s, r))$ and $q(r) = m(B(t, r))$.
  We have $q(r) \leq 1 - p(r)$. Hence, 
  \begin{multline*}
  2\max\{I_m(s), I_m(t)\}  \geq  \int_0^{\tfrac{\diam(F)}{2}} 
  \sqrt{\log\frac{1}{p(r)}} + \sqrt{\log\frac{1}{q(r)}} \, \ud r 
  \\ \geq 
   \Big(\inf_{p \in [0, 1]} \sqrt{\log \frac{1}{p}} + \sqrt{\log\frac{1}{1-p}} \Big)
  \, \frac{\diam(F)}{2}
  \geq \sqrt{\log 2} \,  \diam(F),
  \end{multline*}
  as the infimum is attained at $p = \tfrac{1}{2}$.
  Consequently, by definition $\diam(F) \leq \frac{2}{\sqrt{\log 2}} \Gamma(F)$.
  Combining the previous two displays with~\cref{eqn:liu-functional-controls-gamma}, and taking $a = 10$, yields:
  \[
  \LiuFunctional{F} 
  \geq
  \half \, \sup_{a > 1} c_1(a) \Big(1 - c_2(a) \frac{2}{\sqrt{\log  2}}\Big) 
  \Gamma(F)
  > \frac{17}{800}\,  \Gamma(F). \qedhere
\]
\end{proof}

\bibliographystyle{abbrv}
\bibliography{references}
\end{document}